\documentclass[12pt,a4wide]{article}
\usepackage{amsmath, amsfonts, amssymb, amsthm, euscript, amscd, latexsym, bm, mathrsfs}
  \usepackage[pdftex,colorlinks=true,
                       pdfstartview=FitV,
                       linkcolor=blue,
                       citecolor=blue,
                       urlcolor=blue,
           ]{hyperref}
 \usepackage[dvips]{color}
 \usepackage{titlesec}
\def\eee#1{ \begin{equation} #1 \end{equation} }
\def\aa#1{ \begin{align*} #1 \end{align*} }
\def\aaa#1{ \begin{align} #1 \end{align} }
\def\mm#1{ \begin{multline*} #1 \end{multline*} }
\def\mmm#1{ \begin{multline} #1 \end{multline} }

\newtheorem{thm}{\sc Theorem}
\newtheorem{lem}{\sc Lemma}
\newtheorem{cor}{\sc Corollary}
\newtheorem{df}{\sc Definition}

\newcommand{\ts}{\textstyle}
\newcommand{\sss}{\scriptscriptstyle}
\newcommand{\sst}{\scriptstyle}

\newcommand{\eps}{\varepsilon}
\newcommand{\pl}{\partial}
\newcommand{\gt}{\geqslant}
\newcommand{\lt}{\leqslant}

\newcommand{\x}{\times}
\newcommand{\sub}{\subset}

\newcommand{\bu}{\ts\bigcup\limits}
\newcommand{\dl}{\delta}

\newcommand{\gm}{\gamma}
\newcommand{\Gm}{\Gamma}

\newcommand{\sg}{\sigma}

\newcommand{\dd}{\diagdown}
\newcommand{\om}{\omega}
\newcommand{\Om}{\Omega}
\newcommand{\mc}{\mathcal}
\newcommand{\mf}{\mathfrak}

\newcommand{\td}{\tilde}
\newcommand{\mto}{\mapsto}
\newcommand{\lan}{\langle}
\newcommand{\ran}{\rangle}
\newcommand{\<}{\langle}
\renewcommand{\>}{\rangle}

\newcommand{\bl}{\bigl(}
\newcommand{\br}{\bigr)}

\newcommand{\fdot}{\,\cdot\,}
\newcommand{\C}{{\rm C}}

\newcommand{\de}{\left.\right/}

\newcommand{\nn}{n^{_{\pl V}}}
\newcommand{\nm}{n^{_{\pl V}}(x)}

\newcommand{\emp}{\varnothing}

\newcommand{\h}{{\sss H}}
\newcommand{\we}{\wedge}
\newcommand{\co}{\,\text{\LARGE $\lrcorner$}\,}

\DeclareMathOperator{\ind}{\mathbb I}
\def\Rnu{{\mathbb R}}

\def\Nnu{{\mathbb N}}

\def\ffi{\varphi}

\def\suml{\sum\limits}
\def\intl{\int\limits}
\titleformat{\section}[hang]{\large\bfseries}{\thesection.}{1ex}{}{}
\titleformat{\subsection}[hang]{\normalsize\bfseries}{\thesubsection}{2ex}{}{}

\begin{document}

\author{Evelina Shamarova}
\date{}
\title{Differential forms on locally convex spaces
and the Stokes formula}

\maketitle

\vspace{-12mm}

 {\small
  \begin{center}
  \begin{tabular}{l}
  Grupo de F\'isica Matem\'atica, Universidade de Lisboa. \\
  {E-mail: \href{mailto:evelina@cii.fc.ul.pt}{evelina@cii.fc.ul.pt}}
  \end{tabular}
 \end{center}
 }

%%%%%%%%%%%%%%%%%%%%%%%%%%%%%%%%%%%%%%%%%%%%%%%%%%%%%%%%%%%%%%%%%%%%%%%%%%%%%%%%
%%%%%%%%%%%%%%%%%%%%%%%%%%%%%%%%%%%%%%%%%%%%%%%%%%%%%%%%%%%%%%%%%%%%%%%%%%%%%%%%
% \vspace{-1mm}
\begin{abstract}
 We prove a version of the Stokes formula
 for differential forms on locally convex spaces
 announced in~\cite{WS}.
 The main tool used for proving
 this formula
 is the surface layer theorem proved
 in the paper \cite{mine} by the author.
 Moreover, for differential forms
 of a Sobolev-type class relative to a differentiable
 measure~\cite{AvSF}, we compute the operator
 adjoint to the exterior differential
 in terms of standard operations of calculus of
 differential forms and the logarithmic derivative.
 Previously, this connection was established
 under essentially stronger assumptions on the space~\cite{Sm},
 the measure~\cite{Shig}, or smoothness
 of differential forms~\cite{Sham}. See also~\cite{WLS}.
 \end{abstract}

 \section{Calculus on a Sobolev-type class of  dif\-fe\-ren\-tial forms
  on a locally convex space}
  Let $H$ be a Hilbert space
 with the scalar product
 $(\,\cdot\,,\,\cdot\,)$,
 $\{e_n\}_{n=1}^\infty$be an orthonormal basis of $H$. Let
 $\Gm(n)$
  denote the set of increasing sequences of natural numbers
  of length
 $n\in\Nnu$, $\Gm(0)=\{0\}$. If
 $\gm_1\in\Gm(n)$, $\gm_2\in\Gm(m)$, $n\gt m$,
 then we consider the sequences
 $\gm_1\cup\gm_2$ and $\gm_1\dd\gm_2$ as elements of
 $\Gm(m+n)$ and $\Gm(n-m)$ respectively,
 putting them in the increasing order, if necessarily.
  For every $n\in\Nnu$ and
 $\gm=(i_1,\dots,i_n)\in\Gm(n)$, the symbol $e_\gm$ denotes
 $e_{i_1}\we\dots\we e_{i_n}$, $e_0=1$, where the vectors $e_{i_1},\dots,
 e_{i_n}$ are considered as linear continuous functionals
 on $H$
 (the operation $\we$ is defined, for example, in~\cite{Car}).
 By
 $L_n(H)$
 we denote the space of antisymmetric
 \mbox{$n$-li}near
 Hilbert-Schmidt functionals.
 Note that
 $L_n(H)$
 is a Hilbert space with the scalar product
  $(g_1,g_2)_n=\suml_{i_1<\dots<i_n}g_1(e_{i_1},\dots,
 e_{i_n})\, g_2(e_{i_1},\dots, e_{i_n})$,
  and
 $\{e_\gm\}_{\gm\in\Gm(n)}$ is the orthonormal basis in $L_n(H)$.
 Let $\|\fdot\|_n$ denote the norm which corresponds to $(\fdot,\fdot)_n$.
 Let us show that if $f\in L_n(H)$, $g\in L_m(H)$, then $f\we g\in
  L_{m+n}(H)$.
 Indeed, let $e_\gm = e_{i_1}\we\dots\we e_{i_n}$, and let $g(e_\gm)$
 denote $(g,e_\gm)_n = g(e_{i_1}, \dots , e_{i_n})$.
%  denote the sequence of vectors
%  $e_{i_1},\dots,e_{i_n}$ also by the symbol
%  $e_\gm$ for simplicity.
    We obtain:
  \mm{
  \suml_{\gm\in\Gm(m+n)}(f\we g)^2(e_\gm)=
  \suml_{\gm\in\Gm(m+n)}\Bigl(\suml_{\gm_1\in\Gm(n)}
  \eps(\sg)\,f(e_{\gm_1})\,g(e_{\gm\dd\gm_1})\Bigr)^2 \\
  \lt
  C_{m+n}^n \suml_{\gm\in\Gm(m+n)}\suml_{\gm_1\in\Gm(n)}
  f^2(e_{\gm_1})\, g^2(e_{\gm\dd\gm_1})=
  C_{m+n}^n \suml_{
  \substack{
  \gm_1\in\Gm(n),\\
  \gm_2\in\Gm(m):\\
  \gm_1\cup\gm_2=\gm
  }}
 f^2(e_{\gm_1})\,g^2(e_{\gm_2})\\
  \lt C_{m+n}^n
 \suml_{\gm_1\in\Gm(n)}f^2(e_{\gm_1})
 \suml_{\gm_2\in\Gm(m)}g^2(e_{\gm_2})<\infty.
 }
 Analogously to the finite dimensional case~\cite{Str},
 for elements
  $f\in L_m(H)$ and
 $g\in L_n(H)$, $m>n$, one can define the element $g\co f\in
 L_{m-n}(H)$ by the formula $(g\co f,h)_{m-n}=(f,g\we h)_n$,
 which holds for all $h\in L_{m-n}(H)$.
 Let us show that the operation $\co$ is well defined for
  $f\in L_m(H)$ and $g\in L_n(H)$.
 Specifically, we have to show that
 \[
 \suml_{\gm\in\Gm(m-n)}(g\co f,e_\gm)^2_{m-n}=
 \suml_{\gm\in\Gm(m-n)}(f,g\we e_\gm)_m^2<\infty.
 \]
 Note that
 \[
 (g\we e_\gm)(e_{\gm'})=\sum_{
  \substack{
  \gm_1\in\Gm(n),\\
  \gm_2\in\Gm(m-n):\\
  \gm_1\cup\gm_2=\gm
  }
  }
 \eps(\sg)\,g(e_{\gm_1})\,e_\gm(e_{\gm_2})=
 \begin{cases}
 0,   &\gm\not\sub\gm'\\
 \eps(\sg)g(e_{\gm'\dd\gm}), &\gm\sub\gm'
 \end{cases}
 \]
 where $\eps(\sg)$ is the permutation parity of $\sg$.
 Also, we used here
 the definition of the exterior multiplication~\cite{Car}.
 Taking into account the latter relation, we obtain:
  \[
  \suml_{\gm\in\Gm(m-n)}(f,g\we e_\gm)_m^2= \suml_{\gm\in\Gm(m-n)}
  \Bigl(\suml_{\gm'\in\Gm(m)}f(e_{\gm'})\,(g\we e_\gm)(e_{\gm'})
  \Bigr)^2=
  \]
  \mm{
  =\suml_{\gm\in\Gm(m-n)}\Bigl(
  \sum_{\substack{
  \gm'\in\Gm(m):\\
  \gm\sub\gm'
  }
  }
  \eps(\sg)\,f(e_{\gm'})\,
  g(e_{\gm'\dd\gm}) \Bigr)^2
  \\ \lt \suml_{\gm\in\Gm(m-n)}
  \Bigl(\sum_{\substack{
  \sst \gm_1\in\Gm(n):\\
  \sst \gm\cap\gm_1=\emp
  }}
  f^2(e_{\gm\cup\gm_1})
  \sum_{\substack{
  \gm_1\in\Gm(n):\\
  \gm\cap\gm_1=\emp
  }}
  g^2(e_{\gm_1})\Bigr)
  \\ \lt
  \suml_{\gm_1\in\Gm(n)}g^2(e_{\gm_1})
  \sum_{\substack{
  \gm\in\Gm(m-n),\\
  \gm_1\in\Gm(n):\\
  \gm\cap\gm_1=\emp
  }}
  f^2(e_{\gm\cup\gm_1})\lt C_m^n\suml_{\gm_1\in\Gm(n)}g^2(e_{\gm_1})
  \suml_{\gm_2\in\Gm(m)}f^2(e_{\gm_2})<\infty. }

 Now let $X$ be a locally convex space, and the Hilbert
 space $H$ be a vector subspace of
 $X$.
 \begin{df}
 A mapping $f~:\, X\to L_n(H)$ is called
 a differential form of degree
 $n$ (or differential $n$-form)  on $X$.
 \end{df}
 Note that every differential form
 $f$ can be presented as:
 $f\!=\!\suml_{\gm\in\Gm(n)} f_\gm\,e_\gm$, where $f_\gm$
 are real-valued functions.
 The operations of exterior and interior multiplications
 are defined for differential forms  pointwise.

 Let  $f$ be a differential form of degree $n$.
 \begin{df}
  We say that
  $f$ possesses a differential if its coefficients
  $f_\gm$ are differentiable in each direction
  $e_p$  for  $p\notin\gm$, and for all  $x\in X$,
  $\suml_{\gm\in\Gm(n),\, p\notin\gm}d_{e_p}f_\gm(x)\,e_p\we e_\gm
  \in L_{n+1}(H)$.
 The differential $(n+1)$-form
  \[
  df=\suml_{\gm\in\Gm(n),\, p\notin\gm}d_{e_p}f_\gm\,e_p\we
  e_\gm
  \]
  is called the differential of $f$.
 \end{df}
 \begin{df}
 We say that
 $f$ possesses a codifferential
 if its coefficients
 $f_\gm$
 are differentiable in each direction
 $e_p$ for $p\in\gm$, and for all $x\in X$,
 $\suml_{p\in\gm\in\Gm(n)}d_{e_p}f_\gm(x)\,
 e_p \co e_\gm\in L_{n-1}(H)$.
 The differential $(n-1)$-form
  \[
  \dl f=\suml_{p\in\gm\in\Gm(n)}d_{e_p}f_\gm\, e_p\co e_\gm
  \]
  is called the codifferential of $f$.
  \end{df}
  Let $\mf B_{\sss X}$ be the $\sg$-algebra
  of Borel subsets of the space $X$.
  A measure on $X$ means a $\sg$-additive Hilbert space valued
  function on $\mf B_{\sss X}$.
 \begin{df}
 A $\sg$-additive $L_n(H)$-valued measure on $X$
 is called a differential form
 of codegree~$n$.
 \end{df}
 Every differential form $\om$ of codegree $n$
 can be decomposed as:
  $\om=\suml_{\gm\in\Gm(n)}\om_\gm e_\gm$ where
 $\om_\gm$ are real-valued  $\sg$-additive measures.
  \begin{df}
  Let  $g$ be a bounded differential form of degree $m$,
  $\om$ be a differential form of codegree $n\gt m$
  which is a measure of bounded variation.
  %exterior and interior
  The differential form
  $g\we\om$ of codegree
  $n-m$ defined as
  \[
  (g\we\om)(A)=\int_A g(x)\!\co \om(dx)
  \]
  is called the exterior product of % the differential forms
  $g$ and $\om$.
  \end{df}
  The differential form $g\we\om$ is well defined.
   Indeed, the differential form
  $\om$ can be presented in the form
  $\om=f\cdot|\om|$ (see \cite{Uhl}),
  where  $|\om|$  denotes the variation of $\om$,
  and $f$ is a differential form of degree
  $n$ such that
  $\|f(x)\|_n=1$ for $|\om|$-almost all $x$.
  We have
  \[
  (g\we\om)(A)=\int_A \bigl(g(x)\! \co f(x)\bigr)|\om|(dx).
  \]
  Further,
  \aa{
  &\suml_{\gm\in\Gm(n-m)}((g\we\om)(A),e_\gm)_{n-m}^2=
  \suml_{\gm\in\Gm(n-m)}
  \left(\int_A \bigl(g(x)\!\co f(x),e_\gm\bigr)_{n-m}|\om|(dx)\right)^2\\
  &\lt |\om|(A)\suml_{\gm\in\Gm(n-m)}
  \int_A \bigl(g(x)\!\co f(x),e_\gm\bigr)^2_{n-m}|\om|(dx)\\
  &=|\om|(A)\int_A \|g(x)\!\co f(x)\|_{n-m}^2|\om|(dx)
  \lt C_n^m \,\bigl(|\om|(A)\bigr)^2\sup_x \|g(x)\|_m<\infty.
  }
  \begin{df}
  We say that the differential form
   $\om$ of codegree  $n$ possesses a
  differential if its coefficients $\om_\gm$
  are differentiable in all directions
  $e_p$ for  $p\in\gm$, and
  $\suml_{p\in\gm\in\Gm(n)}d_{e_p}\om_\gm(A)\,
  e_p\co e_\gm\in L_{n-1}(H)$ for all $A\in \mf B_{\sss X}$.
  %\eee{
  %\label{zzhuk}
  %\La_{n-1}(d \om)
  %\eqdef\sqrt{\suml_{\gm\in\Gm(n-1)}
  %\left({\suml_{p\notin\gm}\left|d_{e_p}
  %\om_{\gm\cup p}\right|(X)}\right)^2} <\infty~.
  %}
  The differential form  $d \om$ of codegree $n-1$
 defined by
  \[
  d\om=(-1)^{n-1}\suml_{p\in\gm\in\Gm(n)}d_{e_p}\om_\gm\, e_p\co
  e_\gm,
  \]
  is called the differential of $\om$.
  \end{df}
  \begin{lem}
  \label{lemP}
  Let
   $g$ and $\om$ be differential forms of degree $m$ and codegree
   $n+1 >m$ respectively,
   both possess differentials. Further let $g$ and $\om$ be such that
   $g$ and $dg$ are bounded,  $\om$ and $d\om$
  are of bounded variation.
  Then the differential form
   $g\we\om$ possesses a differential,
   and
  \eee{
  \label{:P}
  d(g\we\om)=g\we d\om+(-1)^n dg\we\om.
  }
  \end{lem}
  The equality~\eqref{:P} can be easily obtained. Indeed, one
  should use the definitions of differentials for
  $g$ and $\om$, the definition of the operation $\we$,
 and compare the coefficients at each $e_\gm$.

\section{The operator adjoint  to the differential}
  Now we compute the operator adjoint to the operator
  $d$ for differential forms of a Sobolev-type class
  relative to a real- or complex-valued
  $\sg$-additive measure on $X$. Let $\mu$ be such a measure.
  We assume that $\mu$ is differentiable
  in each direction
  $e_p$~\cite{AvSF}, and  for all $x$,
  $\|\beta^\mu(x)\|_\h<\infty$, where $ \beta^\mu(x)=\sum_p
  \beta_{e_p}^\mu (x)\,e_p $, and $\beta_{e_p}^\mu$
  is the logarithmic  derivative of the measure
  $\mu$ in the direction
  $e_p$~\cite{AvSF,SmW}.
  Further let the numbers
  $p>1$ and $q>1$ be such that
  $1/p + 1/q=1$. By $\Om_p^n$, we denote the vector space
  of differential $n$-forms
  $f$ satisfying the condition $\int_X\|f(x)\|_n^p\,\mu(dx)<\infty$.
  Define a norm on $\Om_p^n$
  by
  \[
  \|f\|_{n,p}=\left(\int_X \|f(x)\|_n^p\,\mu(dx)\right)^{1\!\de p} \hspace{-2mm}.
  \]
  For elements
  $f=\suml_{\gm\in\Gm(n)}f_\gm(x)\,e_\gm\in\Om_p^n$ and
  $\om=\suml_{\gm\in\Gm(n)}\om_\gm(x)\,e_\gm\in\Om_q^n$,
  we define the bilinear operation
  \eee{
    \label{22}
  \lan\om,f\ran_n=\int_X\bl\om(x),f(x)\br_n\,\mu(dx).
   }
  The integral on the right-hand side exists by
  H\"older's inequality and by the definition of
  $\Om_p^n$.
  By the definition of the scalar product
  $(\cdot,\cdot)_n$ and by Lebesgue's theorem,
  we rewrite~\eqref{22}:
  \aa{
  %\label{23}
  \lan\om,f\ran_n=\suml_{\gm\in\Gm(n)}\int_X f_{\gm}(x)
  \,\om_\gm(x)\,\mu(dx).
  }
  Let
  $A_p^n$
  be the vector subspace of
  $\Om_p^n$,
  consisting of differential forms
  $f$ possessing the codifferential  $\dl f \in \Om_p^{n-1}$,
  satisfying the inequality
  \eee{
  \label{24}
  %\left(\int_X\|\dl f(x))\|^p_{n-1}\,\mu(dx)\right)^{1\!\de p}+
  % \|\dl f\|_{n-1,p} +
  \int_X\|\beta^\mu(x)\|_\h^p\,\|f(x)\|_n^p\,\mu(dx)
  %\right)^{1\!\de p}
  <\infty,
  }
  and such that the following condition (i) is fulfilled.
  Condition (i): for every
  $\gm\in\Gm(n)$, there exists a $\dl>0$ and non-negative functions
  $g_\gm(x)$, $g_{1\gm}(x)$, and $g_{2\gm}(x)$, such that
  $g_\gm(x)$ is
  \mbox{$d_{e_p}\mu$-sum}mable for every $p\notin\gm$,
  $g^2_{1\gm(x)}$ and $g^2_{2\gm}(x)$ are \mbox{$\mu$-sum}mable,
  and for all
  $p\notin\gm$, for $|t|<\dl$,
  $|f_\gm(x+te_p)|<\min\{g_\gm(x),g_{1\gm}(x)\}$ and
  $|d_{e_p}f_\gm(x+te_p)|<g_{2\gm}(x)$.
  On $A_p^n$ we define the norm
  \[
  \|f\|_{A_p^n}= \|f\|_{n,p}+ \|\dl f\|_{n-1,p}
  %\left(\int_X \|\dl f(x)\|_{n-1}\,\mu(dx)\right)^{1\!\de p}
  % \hspace{-4mm}
  +\left(\int_X\|\beta^\mu(x)\|^p_\h\,\|f(x)\|_n^p\,
  \mu(dx)\right)^{1\!\de p} \hspace{-3mm}.
  \]
  Further let
  $B_q^n$ be the vector subspace of
  $\Om_q^n$ consisting of differential forms $\om$
  possessing the differential $d\om \in \Om_q^{n+1}$
  and satisfying the following condition (ii).
  Condition (ii):
   for every
  $\gm\in\Gm(n)$ there exists
  a $\dl>0$ and non-negative functions  $g_\gm(x)$,
  $g_{1\gm}(x)$, $g_{2\gm}(x)$, such that  $g_\gm^2(x)$
  \mbox{$d_{e_q}\mu$-sum}mable for all  $q\in\gm$,
  $g_{1\gm}^2(x)$ and $g_{2\gm}^2(x)$ are \mbox{$\mu$-sum}mable,
  and for all
  $q\in\gm$, for $|t|<\dl$,
  $|\om_\gm(x+tå_q)|<\min\{g_\gm(x),g_{1\gm}(x)\}$ and
  $|d_{e_q}\om_\gm(x+te_q)|<g_{2\gm}(x)$.
  On $B_q^n$ we define the norm
  \[
  \|\om\|_{B_q^n}=\|\om\|_{n,q}+\|d\om\|_{n+1,q}.
  \]
  It is clear that
  $d:\, B^n_q\to\Om_q^{n+1}$ is a linear continuous operator.
  Conditions (i) and (ii) are necessary to satisfy the assumptions
  of the integration by parts formula proved in \cite{AvSF} which we apply
  to compute the operator $d^*$.
% Clearly, these conditions are
%  fulfilled if the functions $f$ and $\om$ and their derivatives
%  in the directions of the basis vectors $e_p$ are bounded.
  % Note that by the inequality
%  $|\lan\om,f\ran_n|\lt\|\om\|_{n,q}\,\|f\|_{n,p}\lt\|\om\|_{n,q}\,\|f\|_{A_p^n}$,
%  every element $f\in A_p^n$ can be presented as a
%  linear continuous functional on
%  $\Om_q^n$  defined by \eqref{22}.
%  Analogously, by the inequality
%  $|\lan\om,f\ran_n|\lt\|\om\|_{B_q^n}\,\|f\|_{n,p}$ each element
%  $f\in\Om_p^n$
%  is a linear continuous functional on
%  $B_q^n$.
%  In this sense, one can speak about the adjoint operator
%  $d^*:\, A_p^{n+1}\to \Om_p^n$.
  \begin{thm}
  For every pair of elements $f\in A_p^{n+1}$ and $\om\in B_q^n$,
  $1/p+1/q=1$, the element $-\beta^\mu\!\co f-\dl f$
  belongs to
  $\Om_p^n$, and
  \[
  \lan d\om,f\ran_{n+1}=\lan\om,-\beta^\mu\!\co f-\dl f\ran_n,
  \]
  i.e. the adjoint operator $d^*:\, A_p^{n+1}\to \Om_p^n$
  is represented by the formula:
   \[
   d^*=-(\beta^\mu\!\co+\dl).
   \]
  \end{thm}
  \begin{proof}
  Let  $\om=\!\!\suml_{\gm\in\Gm(n)}\!\!\om_\gm e_\gm$,
  $f=\!\!\!\suml_{\gm_1\in\Gm(n+1)}\!\!\! f_{\gm_1}e_{\gm_1}$. We
  have
   \mmm{
    \label{die Scheisse}
   \|\beta^\mu(x)\co f(x)\|_n \\
    =\Bigl\|\suml_{p\in\gm_1\in\Gm(n+1)}
   \beta_{e_p}^\mu(x)\,f_{\gm_1}(x)\,e_p\co e_{\gm_1}\Bigr\|_n=
   \Bigl\|\suml_{\gm\in\Gm(n)}\Bigl(\suml_{p\notin\gm}\beta_{e_p}^\mu(x)\,
   f_{\gm\cup p}(x)\Bigr)e_\gm\Bigr\|_n\\
   =\sqrt{\suml_{\gm\in\Gm(n)}\Bigl(\suml_{p\notin\gm}
   \beta_{e_p}^\mu(x)\,f_{\gm\cup p}(x)\Bigr)^2}\lt
   \sqrt{\suml_{\gm\in\Gm(n)}\Bigl(\suml_{p\notin\gm}
   (\beta_{e_p}^\mu(x))^2\suml_{p\notin\gm}f^2_{\gm\cup p}(x)\Bigr)} \\
   \lt \sqrt{\suml_{p=1}^\infty (\beta_{e_p}^\mu(x))^2
   \suml_{\gm\in\Gm(n)}\suml_{p\notin\gm}f^2_{\gm\cup p}(x)}\lt
   \sqrt{n+1}\,\|\beta^\mu\|_\h\,
   \sqrt{\suml_{\gm_1\in\Gm(n+1)}f^2_{\gm_1}(x)} \\
   = \sqrt{n+1}\,\|\beta^\mu(x)\|_\h\,\|f(x)\|_{n+1}.
   }
   From this and~\eqref{24} it follows that $\|\beta^\mu\!\co
   f\|_{n,p}<\infty$.
   By~\eqref{24},
   $\|\dl f\|_{n,p}<\infty$, and hence,
   $\|\beta^\mu\!\co f+\dl f\|_{n,p}<\infty$,
   i.e. $-\beta^\mu\!\co f-\dl f\in \Om_p^n$. Next,
   \mm{
   d\om=\suml_{\gm\in\Gm(n),\,p\in\gm}
   d_{e_p}\om_\gm\,e_p\we e_\gm= \suml_{\gm\in\Gm(n),
   p\notin\gm}(-1)^{k_p-1}\,d_{e_p}\om_\gm\, e_{\gm\cup p}\\
   =\sum_{\gm_1\in\Gm(n+1)}
   \Bigl(\sum_{\substack{
   \gm\sub\gm_1,\\p\in\gm_1\dd\gm}} (-1)^{k_p-1}\,
   d_{e_p}\om_\gm\Bigr)\, e_{\gm_1}, }
  where
  $k_p$ is the number of $p$ in the sequence
  $\gm_1$.
  Applying the integration by parts formula~\cite{AvSF}
  (one can easily verify the conditions under which this
  formula holds), we obtain:
   \aa{
   \lan d\om,f\ran_n
   &=\suml_{\gm_1\in\Gm(n+1)}\intl_X
   f_{\gm_1}(x) \Bigl(\suml_{\substack{\gm\sub\gm_1,\\p\in\gm_1\dd\gm}}
   (-1)^{k_p-1}\, d_{e_p}\om_\gm(x)\Bigr)\,\mu(dx)\\
   &=
   \sum_{\substack{
       \gm_1\in\Gm(n+1),\\
       \gm\sub\gm_1,\,p\in\gm_1\dd\gm
   }}
   (-1)^{k_p-1}\,\intl_X f_{\gm_1}(x)\,
   d_{e_p}\om_\gm(x)\,\mu(dx)\\
   &=\!-\!\!\sum_{\substack{
               \gm_1\in\Gm(n+1),\\
               \gm\sub\gm_1,\,p\in\gm_1\dd\gm
               }}
   (-1)^{k_p-1}\!\intl_X \om_\gm(x)\,(d_{e_p}f_{\gm_1}(x)+
   f_{\gm_1}(x)\,\beta_{e_p}^\mu(x))\,\mu(dx)=
   }
   \aaa{
   \label{30}
   = -\intl_X\suml_{\gm
   \in\Gm(n),\,p\notin\gm} (-1)^{k_p-1}\om_\gm(x)\,
   (d_{e_p}f_{\gm\cup p}(x)+ f_{\gm\cup p}(x)\,
   \beta_{e_p}^\mu(x))\,\mu(dx).
   }
  We changed the order of summation and integration
  when passing to the latter expression in~\eqref{30},
  and applied Lebesgue's theorem. Clearly,
  the sequence of partial sums
  under the last integral sign in~\eqref{30}
  is majorized by an integrable function. This
  follows from the definition of the spaces
  $A^{n+1}_p$ and $B^n_q$,
  from Cauchy-Bunyakovsky-Schwarz's inequality, and
  from the inequality
  \[
  \sqrt{\suml_{\gm\in\Gm(n)}\Bigl(\suml_{p\notin\gm}
  \beta_{e_p}^\mu(x)\,f_{\gm\cup p}(x)\Bigr)^2}\lt
  \sqrt{n+1}\,\|\beta^\mu(x)\|_\h\,\|f(x)\|_{n+1},
  \]
  which was proved, in fact, together with the estimate~\eqref{die Scheisse}.
  By the same argument, all the series
  in~\eqref{30} converge absolutely, and hence the order of summation of these series
  can be chosen arbitrary.
   We rewrite the expressions for
   $\dl f$ and $\beta^\mu\!\co f$:
   \begin{gather*}
   \dl f(x)=\suml_{p\in\gm_1\in\Gm(n+1)}
   d_{e_p}f_{\gm_1}(x)\,e_p\co e_{\gm_1}=
   \suml_{\gm\in\Gm(n)}\Bigl(\suml_{p\notin\gm}
   (-1)^{k_p-1}\,d_{e_p}f_{\gm\cup p}(x)\Bigr)e_\gm,\\
   \begin{split}
   \beta^\mu(x)\co f(x)&=\suml_{p\in\gm_1\in\Gm(n+1)}
   f_\gm(x)\,\beta_{e_p}^\mu(x)\,e_p\co e_{\gm_1}\\
   &=\suml_{\gm\in\Gm(n),\,p\notin\gm}
   (-1)^{k_p-1} f_{\gm\cup p}(x)\beta_{e_p}^\mu(x)\,e_\gm.
   \end{split}
   \end{gather*}
  This and~\eqref{30} imply the statement of the theorem.
  \end{proof}

  \section{The Stokes formula}

  \subsection{Assumtions and notation}

  \renewcommand{\nn}{n^{_{\pl V}}(x_0)}
  \renewcommand{\nm}{n^{_{\pl V}}(x)}
  As before, let
  $X$ be a locally convex space,
 $H$ be its vector subspace which is a Hilbert
 space relative to the scalar product
 $(\,\cdot\,,\,\cdot\,)$, and let
 $\{e_n\}_{n=1}^\infty$ be an orthonormal basis of $H$.
 We assume that
 $H$ is dense in $X$ and  the identical embedding of
 $H$ into $X$ is continuous.
 Let $\Xi_n$,  $n\in N$, denote
 the vector space of bounded differential forms of degree
 $n$ differentiable along $H$ and possessing bounded differentials.
 Let $S_n$  denote the space of
 differential forms of codegree $n$
 which are Radon measures differentiable along $H$.
 Also, we assume that
 the differential forms from $S_n$
 and their differentials are measures
 of bounded variation.
 Let  $\bar S_n$ and $\bar \Xi_n$, $n\in\Nnu$, denote
 pseudo-topological vector spaces
 of linear continuous functionals on $\Xi_n$ and $S_n$,
 respectively.
 We assume that  $\bar S_n$ and
 $\bar \Xi_n$ contain $S_n$ and $\Xi_n$
 as dense subsets.

 In addition to that, we assume that the mapping
 $d:\, \Xi_0 \to \Xi_1$
 can be extended to a continuous mapping
 $\bar \Xi_0 \to \bar \Xi_1$, and
 the mapping $d:\, S_1 \to S_{0}$, to a continuous
 mapping  $\bar S_1\to\bar S_{0}$.
 Further we assume that for every measure
 $\nu\in S_0$, the mapping
 $\Xi_1\to S_1,\; f\mto f\nu$ can be extended
 to a continuous mapping
 $\bar \Xi_1 \to \bar S_1$,
 %We assume that
% for all $m$ and $n$, the mapping
% $\Xi_m\x\Xi_n\to\Xi_{m+n},\,
% (f,g)\mto f\we g$
% can be extended to a continuous mapping
% $\bar\Xi_m\x\bar\Xi_n\to\bar\Xi_{m+n}$,
 and the mapping $\Xi_1\x S_1\to S_{0},\,
 (f,\om)\mto f\we \om$, to a continuous mapping
 $\bar\Xi_1\x\bar S_1\to\bar S_{0}$.
 %$n\gt m$.
 We denote the extended mappings by the same symbols.
 When the functional $f\in \bar S_1$ acts
 on the element $g\in\Xi_1$, we write $\< f,g\>$.
 Further let us assume that every sequence of elements from
 $\Xi_n$, $n=1,2$, converging pointwise to an element from
 $\bar\Xi_n$, converges to this element also
 with respect to the $\bar\Xi_n$-topology.

 Let $V$ be  a domain in $X$  such that
 its boundary $\pl V$
 can be covered with a finite number of
 surfaces $\mc U_i$ of codimension $1$.
 Everywhere  below, a surface of codimension $1$ is the object
 defined in \cite{mine}, p. 552.
 We assume that the indicator $\ind_V$ of the domain $V$
 is an element of the space $\bar\Xi_0$.
 Let the sets
 $\,\mc U_i$ covering  $\pl V$ and their
 intersections possess the property~($*$) formulated in \cite{mine}, p. 559.
 In what follows, we will use the notations introduced in \cite{mine}.
 Here we briefly repeat their meaning:
 $n^{_{\pl V}}: \pl V \to  H$ is the normal vector
 (with respect to the $H$-topology) to $\pl V$;
 let $B\sub \pl V$ be a Borel subset, then
 $B^\eps = \{y\in X: y = x + tn^{_{\pl V}}(x), \, x\in B, |t|<\eps\}$
 is the $\eps$-layer of $B$;
 $\eps_b$ is the maximal number for which $\eps_b$-layers are well defined;
 $\nu^\eps$ is a measure on $\pl V$, $\nu^\eps(B)= \frac{\nu(B^\eps)}{2\eps}$;
 $P_{\sss\pl V}: (\pl V)^{\eps_b} \to \pl V: x+tn^{_{\pl V}}(x) \mto x$
 means the projector of $\eps$-layers to the surface;
 $\nu^{_{\pl V}}$ is the surface measure generated by the measure $\nu$
 (\cite{ugl}, \cite{mine}).
 Rigorous definitions of these objects as well as lemmas proving
 their existence are given in \cite{mine}.
 Note that by Theorem 2 of \cite{mine},
 $\lim_{\eps\to 0}\nu^\eps(\pl V)=\nu^{_{\pl V}}(\pl V)$.

 \subsection{A connection between a measure and the generated surface measure
  in terms of differential forms}

      \begin{thm}
       \label{t2}
       Let
        $\nu\in S_0$.
        Then
       $d\ind_V\cdot \, \nu$
       is an $H$-valued Radon measure on $X$ concentrated on
        $\pl V$.
        Moreover, $n^{_{\pl V}}\cdot\nu^{_{\pl V}}\in \bar S_1$,
        and the measures $\nu$ and $\nu^{_{\pl V}}$ are related
        through the identity:
       \eee{
       \label{main}
       d\ind_V \cdot\,\nu=-n^{_{\pl V}}\cdot\nu^{_{\pl V}}.
       }
      \end{thm}
   Note that by assumption, $d\ind_V \in \bar\Xi_1$
   and $d\ind_V \cdot\, \nu \in \bar S_1$.
%   Also, if  $f\in \Xi_1$, then
%   \aa{
%   \lan n^{_{\pl V}}\cdot\nu^{_{\pl V}}, f\ran_1 =
%   \int_{\pl V} (n^{_{\pl V}}(x),f(x))_1 \, \nu^{_{\pl V}}(dx),
%   }
%   i.e. the right hand side of \eqref{main} is also an element of $\bar S_1$,
%   and hence the relation \eqref{main} makes sense.
 \begin{proof}
 Let $h^\eps: (-\eps_b,\eps_b)\to [0,1]$, $\eps<\eps_b$,
 \[
 h^\eps(\tau)=
 \begin{cases}
 -\frac{\tau}{2\eps}+\frac12~, &\text{if}\;
 \tau\in(-(\eps-\eps^2),\eps-\eps^2),\\
 1, &\text{if}\; \tau\in(-\eps_b,-\eps),\\
 0, &\text{if}\; \tau\in(\eps,\eps_b),
 \end{cases}
 \]
 be $\C^\infty$-smooth  functions such that
 on the intervals
 $(-\eps,-(\eps-\eps^2))$ and
 $(\eps-\eps^2,\eps)$,
 the absolute values of their derivatives
 change monotonically
 from $0$ to $\frac1{2\eps}$, and from
 $\frac1{2\eps}$ to $0$, respectively.
 For $\eps<\eps_b$, we define the functions
 $f^{\eps}: X\to \Rnu$,
 \aaa{
 \label{feps}
 f^{\eps}(x)=
 \begin{cases}
 h^\eps(\tau), &\text{if}
 \; x=y+\tau n^{_{\pl V}}(y),\, y\in\pl V,\,
 \tau\in (-\eps_b,\eps_b),\\
 1, &\text{if} \; x\in V \dd (\pl V)^{\eps_b},\\
 0, &\text{if} \; x\notin V \cup (\pl V)^{\eps_b}.
 \end{cases}
 }
  Let us calculate
 \[
  d_{e_p}f^\eps(x+\tau \nm)=
 \left.\frac{d}{dt}
 f^\eps(x+\tau \nm+te_p)\right|_{t=0}
 \]
  for
 $\tau\in(-(\eps-\eps^2),\eps-\eps^2)$ and
 $x\in\pl V$.
 If $t$ is sufficiently small, then there exist $x_t$ and $\tau_t$,
 such that
 \eee{
 \label{1}
 x+\tau \nm+te_p
 =x_t+\tau_tn^{_{\pl V}}(x_t).
 }
 Let
 $n_p^{_{\pl V}}$ be coordinates of the vector $n^{_{\pl V}}$
 in the basis $\{e_p\}_{p=1}^\infty$.
 Subtracting $x_t$ from the both sides of \eqref{1},
 and multiplying by $\nm$,
 we obtain:
 \[
 (x-x_t,\nm)+\tau+tn_p^{_{\pl V}}(x)
 =\tau_t(\nm,n^{_{\pl V}}(x_t)).
 \]
 Hence,
 \eee{
 \label{2}
 \tau_t=\frac{tn^{_{\pl V}}_p(x)}
 {(\nm,n^{_{\pl V}}(x_t))}
 +\frac{\tau+(x-x_t,\nm)}
 {(\nm,n^{_{\pl V}}(x_t))}.
 }
 Note that
 $x_t=P_{\sss\pl V}(x+\tau \nm+te_p)$.
 We can prove that the derivative
 $\left.\frac{d}{dt}x_t\right|_{t=0}$
 exists with respect to the $H$-topology
 similarly to how it was done in
 the proof of Lemma~6 of~\cite{mine}.
 From the results of~\cite{mine} (Lemmas 1 and 2),
 it follows that the derivative
 $\left.\frac{d}{dt}n^{_{\pl V}}(x_t)\right|_{t=0}$
 exists with respect to the $H$-topology as well.
 Taking into account this, we show that
 $\left.\frac{d}{dt}
 (\nm,n^{_{\pl V}}(x_t))\right|_{t=0}=0$
 and $\left.\frac{d}{dt}
 (x-x_t,\nm)\right|_{t=0}=0$.
 The latter identity is obvious since
 $\left.\frac{d}{dt}x_t\right|_{t=0}\in H_x$,
 and $\nm$ is orthogonal to  $H_x$, where
 $H_x$ is the intersection of the tangent space at $x\in \pl V$
 with $H$ (see \cite{mine}).
 Further we have:
 \mm{
 0=\left.\frac{d}{dt}\|n^{_{\pl V}}(x_t)\|^2\right|_{t=0}
 =2\Bigl(\nm,
 \left.\frac{d}{dt}n^{_{\pl V}}(x_t)\right|_{t=0}\Bigr)\\ =
 2\left.\frac{d}{dt}(\nm,
 n^{_{\pl V}}(x_t))\right|_{t=0}.
 }
 From this and from~\eqref{2}, it follows that
 $\left.\frac{d}{dt}\tau_t
 \right|_{t=0}=n_p^{_{\pl V}}(x)$.
 Taking into account that
 $f^\eps(x+\tau \nm+te_p)=
 -\frac{\tau_t}{2\eps}+\frac12$,
 we obtain that
 $d_{e_p}f^\eps(x+\tau \nm)=
 -\frac{n_p^{_{\pl V}}(x)}{2\eps}$, and hence,
 \eee{
 \label{4}
 df^\eps(x+\tau n^{_{\pl V}}(x))=-
 \frac{\nm}{2\eps}.
 }
 For $\tau$ which belongs to one of the intervals
 $(-\eps,-(\eps-\eps^2))$ or
 $(\eps-\eps^2,\eps)$,
 the differential
 $df^\eps(x+\tau \nm)$
 can be calculated in the same way.
 Indeed,
 $f^\eps(x+\tau \nm+te_p)=
 h^\eps(\tau_t)$, and
 \eee{
 \label{3}
 df^\eps(x+\tau \nm)=
 (h^\eps)'(\tau)\nm.
 }
 This implies that
 \aa{
 %\label{5}
 \|df^\eps(x)\|<\frac1{2\eps} \quad
 \text{for all}\; x\in X.
 }
 Note that as $\eps\to 0$, $f_\eps \to \ind_V$
 pointwise, and hence with respect to the $\bar \Xi_0$-topology.
 By assumption, $df_\eps \to d\ind_V$ in the $\bar \Xi_1$-topology.
 Let $g\in \Xi_1$.
 We have:
 \mm{
 %\label{12}
 \<d\ind_V\cdot\nu,g\>
 = \lim_{\eps\to 0} \<df_\eps \cdot\, \nu, g\>
 = \lim_{\eps\to 0}
 \int_{X}(df_\eps(x),g(x))\nu(dx) \\
 =\lim_{\eps\to 0}
 \int_{(\pl V)^{\eps}}(df_\eps(x),g(x))\nu(dx).
 }
 Let
 $x\in \pl V$.
 The function $[0,\eps_b) \to \Rnu$, $t\mto (\nm,g(x+t\nm))$
 is differentiable.
 By assumption,
 $g$ has a bounded derivative, say by a constant $M$,
 along $H$.
 For all
 $x\in\pl V$, $t \in (-\eps,\eps)$, we obtain:
 \mm{
 |(\nm,g(x+t\nm))-(\nm,g(x))|\\
 \lt \Bigl|\Bigl(\nm,\frac{d}{dt}g(x+t\nm)
 |_{t=t_0}\Bigr)\Bigr|\cdot t
 \lt \|g'(x+t_0\nm)\nm\|\cdot t\\
 \lt \|g'(x+t_0\nm)\|_2\cdot t<M\eps,
 }
 where $t_0<t$.
 Define a function
 $\td g:\, (\pl V)^{\eps_b}\to H$
 in the following way: for
 $x\in\pl V$, $t \in (-\eps_b,\eps_b)$, we set
 $\td g(x+t\nm)=g(x)$.
 Then, taking into account the above sequence of
 inequalities, formulas \eqref{4}, \eqref{3},
 and the definition of
 $h^\eps$, for all $x\in (\pl V)^\eps$
 we obtain that
 \[
 |(df^\eps(x),g(x))-(df^\eps(x),\td g(x))|<\frac{M}2.
 \]
 This implies:
 \aaa{
 \label{8}
 \begin{split}
 \<(d\ind_V)\cdot\nu,g\>&=
 \lim_{\eps\to 0}
 \int_{(\pl V)^{\eps}}(df_\eps(x),\td g(x))\nu(dx)\\
 &=-\lim_{\eps\to 0}\int_{(\pl V)^{\eps}}
 \frac1{2\eps}\,(n^{_{\pl V}}(P_{\sss \pl V}x),
 g(P_{\sss \pl V}x))\nu(dx)\\
 &\quad +
 \lim_{\eps\to 0}\int_{(\pl V)^{\eps}\dd (\pl V)^{\eps-\eps^2}}
 \bigl(df_\eps(x)-\frac1{2\eps}\,n^{_{\pl V}}(P_{\sss \pl V}x),\td g(x)\bigr)\nu(dx)
 \\
 &=-\lim_{\eps\to 0}\int_{\pl V}(\nm,g(x))\nu^\eps(dx).
 \end{split}
 }
 Indeed,
 \aa{
 \int_{(\pl V)^{\eps}}
 \frac1{2\eps}\,(n^{_{\pl V}}(P_{\sss \pl V}x),
 g(P_{\sss \pl V}x))\nu(dx)
 = \int_{\pl V}(\nm,g(x))\nu^\eps(dx),
 }
 and as $\eps\to 0$,
 \mm{
 \Bigl|\int_{(\pl V)^{\eps}\dd (\pl V)^{\eps-\eps^2}}
 (df_\eps(x)-\frac1{2\eps}n^{_{\pl V}}
 (P_{\sss \pl V}x),\td g(x))\nu(dx)\Bigr|\\
 \lt M\,
 \frac{\nu((\pl V)^{\eps}\dd (\pl V)^{\eps-\eps^2})}{\eps}=
 2\,\Bigl(
 \frac{\nu((\pl V)^\eps)}{2\eps}-
 \frac{\nu((\pl V)^{\eps-\eps^2})}{2(\eps-\eps^2)}
 \frac{\eps-\eps^2}{\eps}\Bigr)\to 0.
 }
 Further, we fix an arbitrary
 $\sg>0$, and let
 $\sg'=\frac{\sg}{2(M+\nu^{\sss \pl V}(\pl V))}$.
 Since
 $\nu^{_{\pl V}}$
 is a Radon measure
 (see \cite{ugl}, \cite{ugl4}),
 then there exists a compact
 $K_\sg\sub \pl V$, such that
 $\nu^{_{\pl V}}(\pl V\dd K_\sg)<\sg'$.
 For each point
 $x_0\in K_\sg$ we fix a neighborhood $U_{x_0}$,
 which is contained in one of
 $\mc U_i$, possesses the property~($*$)
 formulated in \cite{mine},
 and such that for the function
 $\ffi(x)=(\nm,g(x))$, the inequality
 $|\ffi(x)-\ffi(x_0)|<\sg'$ holds for all
 $x\in U_{x_0}$.
 We choose a finite number of neighborhoods
 $U_x$, $x\in K_\sg$, covering $K_\sg$
 (let them be neighborhoods $U_i$ of points $x_i$),
 and denote their union by
 $O_\sg$.
 It is clear that
 $\nu^{_{\pl V}}(\pl V\dd O_\sg)<\sg'$, and by
 the construction of
 $O_\sg$, there exists the limit
 $\lim_{\eps\to 0}\nu^\eps(O_\sg)
 =\nu^{_{\pl V}}(O_\sg)$.
 Hence the limit
 $\lim_{\eps\to 0}\nu^\eps(\pl V \dd O_\sg)
 =\nu^{_{\pl V}}(\pl V\dd O_\sg)$ exists too.
 Further let
 $B_i=U_i\dd \bu_{j=1}^{i-1}U_j$,
 and $\ffi_\sg:\, O_\sg\to\Rnu$ be such that
 $\ffi_\sg=\suml_i\ffi(x_i)\ind_{B_i}$.
 It is clear that on $O_\sg$,
 $|\ffi(x)-\ffi_\sg(x)|<\sg'$.
 We have:
 \mmm{
 \label{6}
 \lim_{\eps\to 0}\int_{\pl V}\ffi(x)\nu^\eps(dx)
 =\lim_{\eps\to 0}\int_{O_\sg}\ffi_\sg(x)\nu^\eps(dx)\\+
 \lim_{\eps\to 0}\Bigl(
 \int_{O_\sg}(\ffi(x)-\ffi_\sg(x))\nu^\eps(dx)
 +\int_{\pl V\dd O_\sg}\ffi(x)\nu^\eps(dx)\Bigr).
 }
 By the definition of
 $\ffi_\sg$,
 $
 \int_{O_\sg}\ffi_\sg(x)\nu^\eps(dx)=\suml_i
 \ffi(x_i)\nu^\eps(B_i)
 $,
 where the sum contains a finite number of terms.
 We observe that for every set $B_i$,
 $\lim_{\eps\to 0}\nu^\eps(B_i)=\nu^{\sss\pl V}(B_i)$
 by the construction of $B_i$ and by Theorem 2 of \cite{mine}.
 Hence  $\lim_{\eps\to 0}\int_{O_\sg}\ffi_\sg(x)\nu^\eps(dx)
 =\int_{O_\sg}\ffi_\sg(x)\nu^{\sss\pl V}(dx)$.
 The limit of the second term in~\eqref{6} exists
 by the existence of the two other limits.  We continue~\eqref{6}:
 \mm{
 \lim_{\eps\to 0}\int_{\pl V}\ffi(x)\nu^\eps(dx)=
 \int_{\pl V}\ffi(x)\nu^{\sss\pl V}(dx)\\-
 \int_{\pl V\dd O_\sg}\ffi(x)\nu^{\sss\pl V}(dx)-
 \int_{O_\sg}(\ffi(x)-\ffi_\sg(x))\nu^{\sss\pl V}(dx)\\+
 \lim_{\eps\to 0}\Bigl(
 \int_{O_\sg}(\ffi(x)-\ffi_\sg(x))\nu^\eps(dx)
 +\int_{\pl V\dd O_\sg}\ffi(x)\nu^\eps(dx)\Bigr).
 }
 Let us estimate the last three terms. We have:
 \[
 \Bigl|
 \int_{O_\sg}(\ffi(x)-\ffi_\sg(x))\nu^\eps(dx)
 +\int_{\pl V\dd O_\sg}\ffi(x)\nu^\eps(dx)\Bigr|\lt
 \sg'\nu^\eps(O_\sg)+M\nu^\eps(\pl V\dd O_\sg).
 \]
 Passing to the limit as $\eps\to 0$ in the both sides of this
 inequality, and taking into account that the limit on the
 left-hand side exists, we obtain:
 \mm{
 \Bigl|\lim_{\eps\to 0}\Bigl(
 \int_{O_\sg}(\ffi(x)-\ffi_\sg(x))\nu^\eps(dx)
 +\int_{\pl V\dd O_\sg}\ffi(x)\nu^\eps(dx)\Bigr)\Bigr|
 \lt \sg'\nu^{\sss\pl V}(\pl V)+M\sg'.
 }
 Analogously, we prove the two other estimates:
 \[
 \Bigl|\int_{\pl V\dd O_\sg}\ffi(x)\nu^{\sss\pl V}(dx)\Bigr|
 < M\sg',\qquad
 \Bigl|
 \int_{O_\sg}(\ffi(x)-\ffi_\sg(x))\nu^{\sss\pl V}(dx)
 \Bigr|< \sg'\nu^{\sss\pl V}(\pl V).
 \]
 From this it follows that
 \[
 \Bigl|
 \lim_{\eps\to 0}\int_{\pl V}\ffi(x)\nu^\eps(dx)
 -\int_{\pl V}\ffi(x)\nu^{\sss\pl V}(dx)
 \Bigr|<\sg.
 \]
 Since $\sg > 0$ was chosen arbitrary, we conclude that
 \[
 \lim_{\eps\to 0}\int_{\pl V}(\nm,g(x))\nu^\eps(dx)
 =\int_{\pl V}(\nm,g(x))\nu^{\sss\pl V}(dx).
 \]
 Together with~\eqref{8} this implies that
 \[
 \<(d\ind_V)\cdot\nu,g\>=-\int_{\pl V}(\nm,g(x))\nu^{\sss\pl V}(dx)
 \]
 which is equivalent to~\eqref{main}. The theorem is proved.
 \end{proof}
 \begin{cor}
 \label{cor1}
 Let $\om = \sum_{p=1}^\infty \om_p e_p \in S_1$,
 and $\om^{_{\pl V}} = \sum_{p=1}^\infty \om_p^{_{\pl V}} e_p \in \bar S_1$,
 where $\om_p^{_{\pl V}}$ are the surface measures generated
 by the measures $\om_p$.
 Then
 $n^{_{\pl V}}\in \bar \Xi_1$, and the measures $\om$ and $\om^{_{\pl V}}$
 are related through the identity:
 \aa{
 d\ind_V\we\, \om=-\,n^{_{\pl V}}\we\,\om^{_{\pl V}}.
 }
 \end{cor}
 \begin{proof}
 Let us prove that $n^{_{\pl V}}\in \bar \Xi_1$.
 Indeed, $n^{_{\pl V}}$ originally defined on $\pl V$ can be extended
 to $X$ by setting $n^{_{\pl V}}(x) = 0$ for $x\notin \pl V$.
 Let us consider the functions $f_\eps$ defined by \eqref{feps}.
 By \eqref{4}, $-2\eps\, df_\eps$ converges to $n^{_{\pl V}}$
 pointwise, and hence with respect to the $\bar \Xi_1$-topology by assumption.
 Note that by assumption, the operation $\we$
 can be extended from $\Xi_1 \x S_1$ to
 $\bar \Xi_1 \x \bar S_1$ so that $(f,\om)\mto f\we\om$
 is a continuous mapping $\bar \Xi_1 \x \bar S_1 \to \bar S_0$.
 Applying
 Theorem~\ref{t2} to each pair of
 real-valued measures $\om_p$ and
 $\om_p^{_{\pl V}}$ we obtain:
 \aa{
  d\ind_V\we\, \om = \sum_{p=1}^\infty d\ind_V \cdot \,\om_p \we\, e_p
  = - \sum_{p=1}^\infty (n^{_{\pl V}}\cdot\,\om^{_{\pl V}}_p) \we\, e_p
  = -\,n^{_{\pl V}}\we\,\om^{_{\pl V}}.
 }
 \end{proof}

 \subsection{Derivation of the Stokes formula}
 \begin{df}
 \label{def7}
 Let $\om\in S_1$ and $\om^{_{\pl V}}\in \bar S_1$.
 We define the integral of $\om$
 over the surface
 $\pl V$  by the identity:
 \[
 \int_{\pl V}\om=\int_{\pl V}(n^{_{\pl V}}(x),\om^{_{\pl V}}(dx)).
 \]
 \end{df}
 \begin{thm}[The Stokes formula]
 Let $\om\in S_1$ and $\om^{_{\pl V}}\in \bar S_1$, then
 \[
 \int_{\pl V}\om=\int_Vd\om.
 \]
 \end{thm}
 \begin{proof}
 Corollary \ref{cor1} and Definition \ref{def7} imply:
 \aa{
 \int_{\pl V}\om = \int_{\pl V}(n^{_{\pl V}}(x),\om^{_{\pl V}}(dx))
 = -\<d\ind_V\we\, \om, 1\>.
 }
 Let us consider again the functions $f_\eps$ defined by \eqref{feps}.
 We proved that $f_\eps \to \ind_V$ pointwise
 and in the $\bar \Xi_0$-topology,
 and that $df_\eps \to d\ind_V$ in the $\bar \Xi_1$-topology.
 %Note that by assumption, the operator $d: S_1\to S_0$ can be extended
% to a continuous mapping $\bar S_1 \to \bar S_0$, and the
% operator $d: \Xi_0 \to \Xi_1$ to a continuous mapping $\bar \Xi_0 \to \bar \Xi_1$.
% This allows us to extend
% the formula \eqref{:P} in Lemma \ref{lemP} to the elements from $\bar \Xi_0$
% and $\bar S_1$.
 By Lemma \ref{lemP},
 \[
 0= d(f_\eps\we\,\om)(X)=(df_\eps\we\,\om)(X)+(f_\eps\we
 \,d\om)(X).
 \]
 Hence,
 \[
 \int_Vd\om= \lim_{\eps\to 0} (f_\eps \we\, d\om)(X)=
  -\lim_{\eps\to 0}(df_\eps\we\,\om)(X) = - \<d\ind_V\we\, \om, 1\>
 =\int_{\pl V}\om.
 \]
 The theorem is proved.
 \end{proof}

\end{document}